\documentclass[12pt,oneside]{amsart}
\usepackage[cp1251]{inputenc}
\usepackage{amsmath,amssymb,amsthm,mathrsfs}
\usepackage[mathscr]{eucal}
\usepackage{graphicx, epsfig}

\theoremstyle{plain}
\newtheorem{theorem}{Theorem}
\newtheorem{lemma}{Lemma}

\author{Evgeny~Fominykh}
\author{Andrei~Malyutin}
\author{Ekaterina~Shumakova}

\address{St.~Petersburg University and St.~Petersburg Department of V.~A.~Steklov Mathematical Institute}
\email{fominykh@gmail.com}

\address{St.~Petersburg Department of V.~A.~Steklov Mathematical Institute and St.~Petersburg University}
\email{malyutin@pdmi.ras.ru}

\address{Chelyabinsk State University and St.~Petersburg University}
\email{shumakova\_kate@mail.ru}

\thanks{This work is supported by the Russian Science Foundation under grant 19-11-00151.}

\title{$3$-manifolds represented by $4$-regular graphs with three Eulerian cycles}

\begin{document}

\maketitle

We construct and study a new class of compact hyperbolic $3$-manifolds with totally geodesic boundary. 
This class extends the line of research developed in \cite{FMP03a, VesTurFom16, VesFom12, FomShum21}
and exhibits a number of remarkable properties.
The members of this class are described by triples of Eulerian cycles in $4$-regular graphs.
Two Eulerian cycles are said to be \emph{compatible} if no pair of adjacent edges are consecutive in both cycles.
If a $4$-regular graph $G$ contains 
a triple~$\theta$ of pairwise compatible Eulerian cycles,
we say that $G$ is \emph{$3$-Eulerian} and $\theta$ is a \emph{framing} of~$G$.
Each finite $3$-vertex-connected simple $4$\nobreakdash-regular graph is $3$-Eulerian~\cite{J91}.
Let $G$ be a $3$-Eulerian graph, and let~$\theta$ be a framing of~$G$. 
A~\emph{polyhedral realization} of the pair $(G,\theta)$ is a $2$-dimensional
polyhedron $P(G,\theta)$ obtained from $G$ by attaching a $2$-cell along each cycle in $\theta$.

\begin{lemma}
    $P(G, \theta)$ is a special spine of a $3$-manifold with nonempty boundary.
\end{lemma}
\begin{proof}
    As follows from the theory of special polyhedra \cite[chapter 1]{Mat07}, it suffices to verify that the boundary curve of each $2$-component $\xi$ of $P=P(G, \theta)$ is orientation preserving on the closed surface $P\setminus\xi$. The latter follows from the fact that the remaining $2$-components of $P$ give a checkerboard coloring on $P\setminus \xi$, since their boundary curves are Eulerian cycles. This completes the proof.
\end{proof}

    It is known~\cite[Theorem 1.1.17]{Mat07} that each special spine of a $3$-manifold~$M$ with boundary determines~$M$ uniquely. 
    Denote by~$M(G,\theta)$ the manifold (with boundary) determined by $P(G,\theta)$.

\begin{theorem}
    Let $G$ be a $3$-Eulerian graph with $n\ge 4$ vertices, and let~$\theta$ be a framing of~$G$. Then
    \begin{enumerate}
        \item $M(G,\theta)$ is hyperbolic with totally geodesic connected boundary.
        \item Matveev complexity of $M(G, \theta)$ equals $n$.
        \item The topological ideal triangulation $T(G, \theta)$ dual to $P(G, \theta)$ is minimal.
    \end{enumerate}
\end{theorem}
\begin{proof}
    (1) The truncated triangulation~$T^*$ of $M = M(G,\theta)$ dual to $P=P(G,\theta)$ consists of $n$ truncated tetrahedra (according to the number of true vertices of $P$) and has exactly three edges (according to the number of $2$-components of $P$) such that each edge of~$T^*$ is incident to exactly $2n$ dihedral angles of the tetrahedra. The last condition (since $n\ge 4$) allows one to realize topological truncated tetrahedra as congruent regular truncated hyperbolic tetrahedra, which gives $M$ a hyperbolic structure. A detailed description of this trick is given, e.\,g., in ~\cite{FMP03a}. To prove that $\partial M$ is connected (which also holds for $n\in\{2,3\}$) consider the graph~$\Gamma$ whose vertices are the components of $\partial M$ and whose edges correspond to the edges of~$T^*$. Analyzing the correspondence between the edges of a tetrahedra in~$T^*$ and the edges of~$\Gamma$ we get that $\Gamma$ is a wedge of three circles.   
    
    (2)-(3) The proof boils down to demonstrating that~$M$ has no special spine with less than $n$ (true) vertices. The sufficiency follows by (1), since every compact hyperbolic $3$\nobreakdash-manifold with totally geodesic boundary has a special spine whose number of vertices is equal to its Matveev complexity~\cite[Theorem 2.2.4]{Mat07}. Assume that $M$ has a special spine $Q$ with $n'<n$ vertices. The equality of the Euler characteristics $\chi(Q)=\chi(M)=\chi(P)$ implies that the number of $2$-components in~$Q$ is less than that in~$P$. From this we deduce that $Q$ contains at most one closed surface and $P$ contains three ones. This is impossible because $H_2(Q;\mathbb{Z}_2)=H_2(M;\mathbb{Z}_2)=H_2(P;\mathbb{Z}_2)$ while the number of closed surfaces contained in an arbitrary special polyhedron $S$ equals the number of nontrivial elements in the homology group $H_2(S;\mathbb{Z}_2)$.
\end{proof}

\begin{theorem}
    Let $G$ and $G'$ be $3$-Eulerian graphs with $n\ge 4$ and $n'\ge 4$ vertices and with framings $\theta$ and $\theta'$, respectively. 
    Then the manifolds $M(G,\theta)$ and $M(G',\theta')$ are homeomorphic if and only if the framed graphs $(G,\theta)$ and $(G',\theta')$ are isomorphic.
\end{theorem}

\begin{proof}
    The proof follows from assertion~(1) of Theorem~1 and the Mostow rigidity theorem because, in the above construction with congruent regular truncated hyperbolic tetrahedra, 
    the polyhedron $P(G,\theta)$ is homeomorphic to the cut locus of the corresponding hyperbolic manifold, and thus $P(G,\theta)$ is uniquely determined by the manifold.
\end{proof}

We denote the class of all manifolds of the form $M(G,\theta)$ with complexity $n$ by $\mathscr{M}_n$.

\begin{theorem}
For each sufficiently large~$n\in\mathbb{N}$ we have
 $$n! < |\mathscr{M}_n| < n!\,4^n.$$
\end{theorem}

\begin{proof}
    Since each Eulerian graph $G$ is uniquely determined by the sequence of vertices of any Eulerian path in~$G$, 
    it follows that 
    the number of connected $4$-regular $n$-vertex graphs with a distinguished Eulerian path is $\frac{(2n)!}{n!\,2^n}$.
    Since an Eulerian cycle occurs in at most~$2^{n-1}$ framings, we have 
    $$\lvert \mathscr{M}_n\rvert\le\frac{(2n)!}{n!\,2}<n!\,4^n.$$ 
    
    The results of \cite{War81} and \cite{Bol82} imply that both the number of 
    $4$-vertex-connected asymmetric (i.\,e., having no nontrivial automorphisms) 
    simple $4$-regular $n$-vertex unlabeled graphs and the number of all simple $4$-regular $n$\nobreakdash-vertex unlabeled graphs are asymptotic to 
    $$\frac{e^{-15/4} (4n)!}{(96)^n (2n)! n!}.$$
    All finite $3$-vertex-connected simple $4$\nobreakdash-regular graphs are $3$-Eulerian~\cite{J91}.
    Analysis of options shows that for each framing~$\theta$ of an arbitrary $3$\nobreakdash-Eulerian graph~$G$ and for each vertex~$v$ in~$G$ there is at least one way to change the cycles of~$\theta$ at~$v$ so as to obtain a framing of~$G$ that is distinct from~$\theta$.
    This implies that each asymmetric $3$-Eulerian $n$-vertex graph admits at least $2^{n-1}$ framings.
    Combining this with the Stirling formula we deduce that for any $C<1$ and for all sufficiently large $n$ we have 
    $$|\mathscr{M}_n|> n! C 4^n/(3^{n}2\pi n e^{\frac{15}{4}}\sqrt{2}).$$
\end{proof}


\begin{thebibliography}{9}
\bibitem{FMP03a}
R. Frigerio,  B. Martelli, C. Petronio, 
Complexity and Heegaard genus of an infinite class of compact $3$-manifolds, 
\textit{Pacific J. Math.}, 
\textbf{210}:2 (2003), 283--297.

\bibitem{VesTurFom16}
A.Yu. Vesnin, V.G.Turaev, E.A. Fominykh, Complexity of virtual $3$-manifolds, 
\textit{Sbornik: Math.}, \textbf{207}:11 (2016), 1493--1511.

\bibitem{VesFom12}
A.Yu. Vesnin, E.A. Fominykh, On complexity of three-dimensional hyperbolic manifolds with geodesic boundary, 
\textit{Siberian Math. J.}, \textbf{53}:4 (2012), 625–634.

\bibitem{FomShum21}
E. Fominykh, E. Shumakova, Poor ideal three-edge triangulations are minimal, 
\textit{arXiv:2105.05110}, (2021), 10 pp.

\bibitem{J91}
B.~Jackson, A characterisation of graphs having three pairwise compatible Euler tours, 
\textit{J. Combin. Theory Ser. B}, \textbf{53}:1 (1991), 80--92. 

\bibitem{Mat07}
S. Matveev, 
\textit{Algorithmic topology and classification of $3$-manifolds}, 2-nd ed. 
Algorithms and Computation in Mathematics, 9. 
Springer, Berlin, 2007.
xiv+492 pp.

\bibitem{War81}
N. C. Wormald,
The asymptotic connectivity of labelled regular graphs,
\textit{J. Combin. Theory Ser. B}, 
\textbf{31} (1981), 156--167.

\bibitem{Bol82}
B. Bollob\'as, 
The asymptotic number of unlabelled regular graphs, 
\textit{J. Lond. Math. Soc.},
\textbf{26} (1982),
201--206.

\end{thebibliography}
\end{document}